\documentclass[a4paper,reqno,10pt]{amsart}

\usepackage[all]{xy}
\usepackage[english]{babel}
\usepackage{amssymb}
\usepackage{graphicx}
\usepackage{amsmath,amsthm,amssymb,xypic,verbatim}

\newcommand{\p}{\mathbb{P}}

\newcommand{\N}{\mathbb{N}}

\newcommand{\C}{\mathbb{C}}

\renewcommand{\o}{{\mathcal O}}

\newcommand{\I}{\mathcal{I}}

\newcommand{\h}{{\mathcal H}}

\newcommand{\T}{\mathbb{T}}

\newcommand{\Bs}{\mathop{\rm Bs}\nolimits}

\newcommand{\codim}{\mathop{\rm codim}\nolimits}

\newcommand{\Sec}{\mathop{{\rm{\mathbb S}ec}}\nolimits}
\newcommand{\Span}[1]{\langle#1\rangle}

\theoremstyle{plain}

\newtheorem{thm}{Theorem}
\newtheorem{prop}[thm]{Proposition}
\newtheorem{lem}[thm]{Lemma}   

\newtheorem{claim}{Claim}
\theoremstyle{definition}

\newtheorem{dfn}[thm]{Definition}
\newtheorem{rem}[thm]{Remark}
\newtheorem{remark}[thm]{Remark}

\newcommand{\mult}[0]{\operatorname{mult}}

\begin{document}
\title{Tangential weak defectiveness and  generic identifiability}

\author[Alex Casarotti]{Alex Casarotti}
\address{\sc Alex Casarotti\\ Dipartimento di Matematica e Informatica, Universit\`a di Ferrara, Via Machiavelli 30, 44121 Ferrara, Italy}
\email{csrlxa@unife.it}

\author[Massimiliano Mella]{Massimiliano Mella}
\address{\sc Massimiliano Mella\\ Dipartimento di Matematica e Informatica, Universit\`a di Ferrara, Via Machiavelli 30, 44121 Ferrara, Italy}
\email{mll@unife.it}

\subjclass[2020]{Primary 14N07; Secondary 14N05, 14N15, 14M15, 15A69, 15A75}
\keywords{Secant varieties, secant defectiveness, weak defectiveness, tangential weak defectiveness, identifiability}
\date{\today}
\begin{abstract}
 We investigate the uniqueness of decomposition of general tensors $T\in {\mathbb
  C}^{n_1+1}\otimes\cdots\otimes{\mathbb C}^{n_r+1}$ as a sum
 of tensors of rank $1$.  This is done  extending the theory developed in \cite{Me1} to the framework of non twd varieties. In this way we are able to prove the non
 generic identifiability of infinitely many partially symmetric tensors.
\end{abstract}

\maketitle 

 \section{Introduction}
The decomposition of tensors $T\in {\mathbb
  C}^{n_1+1}\otimes\cdots\otimes{\mathbb C}^{n_r+1}$ as a sum
 of {\it simple} tensors (i.e. tensors of rank $1$) is a central
 problem for many applications from Multilinar Algebra to Algebraic
 Statistics, coding theory, blind signal separation and others, \cite{DDL1},\cite{DDL2},\cite{DDL3},\cite{KADL},\cite{Si}.

 For  statistical inference, it is meaningful  to know if  
 a probability distribution, arising from a
model, uniquely determines the parameters that produced it. 
When this happens, the parameters are called {\it identifiable}.
There are no useful models where all distributions are identifiable. 
Then the notion of {\it generic} identifiability for parametric models  
has been considered for instance in \cite{AMR09} and in \cite{SR12}.
Conditions which guarantee the uniqueness of decomposition, 
for generic tensors in the model, are quite important in the applications. 
When generic identifiability holds, the set of non-identifiable
parameters has measure zero, thus parameter inference is still meaningful.
Notice that many decomposition algorithms converge to {\it one} decomposition,
hence a uniqueness result guarantees that the decomposition found is the chased one. 
We refer to \cite{KB} and  its huge  reference list, for more details.

From a purely theoretical point of view, the study of unique
decompositions, or canonical forms in the early $\rm XX^{\rm th}$ century
dictionary, has connection with both invariant theory, \cite{Hi}, and
projective geometry, \cite{Pa} \cite{Ri}. 
It is already over a decade,   \cite{Me1},  that generic identifiability of symmetric
tensors has shown its close
connection to modern birational projective geometry and especially to
the maximal singularities methods. In a series of papers, \cite{Me1}
\cite{Me2} \cite{GM},  the generic
identifiability problem for symmetric tensors has been completely
solved. 

The present paper is devoted to extend this theory to arbitrary
tensors and can be considered as a first step, similar to \cite{Me1},
in this direction. As for the symmetric case it is expected that
identifiability is very rare and our result support this convincement.

The main tool in \cite{Me1} was the use, after
\cite{CC1},  of non weakly
defective varieties to study identifiability, see Section~\ref{sec:notation}
for all the relevant definitions. Unfortunately it is very hard to
determine the weak defectiveness of general tensors.  This difficulty prevented, for many years, a
straightforward application of the same techniques to them, see
\cite{Fo} for a similar approach in special cases.

 In recent years the notion of tangential weakly defectiveness,
 introduced in \cite{CO1}, has gradually substituted the weak
 defectivity and proved valid to study generic identifiability of
 subgeneric tensors, \cite{CO1} \cite{BDdG} \cite{BC} \cite{BCO}
 \cite{Kr} \cite{CM}. In particular thanks to the main result in
 \cite{CM} for the  generic identifiability we may assume without
 loss of generality the non tangential weakly defectiveness under mild
 numerical assumptions.

 Tangential weakly defectiveness does not behave as
 weakly defectiveness with respect to the maximal singularities
 method. Therefore in this paper we develop tools to plug in maximal singularities
 methods for non tangentially weakly defective varieties.
 In this way we are able to prove
the non identifiability of many partially symmetric tensors.
The main technical result is a study of the nested  singularities of tangential
linear system for non tangentially weakly defective varieties and with
this we are able to prove the following statement on identifiability
of partially symmetric tensors.

\begin{thm}
The generic tensor  $T\in
{\rm Sym}^{d_i}({\mathbb C}^{n_{1}+1})\otimes\cdots\otimes{\rm
  Sym}^{d_r}({\mathbb C}^{n_{r}+1})$ is
not identifiable when $d_i > n_i+1$ for any $1 \leq i \leq r$ and $\lceil\frac{{\prod}\binom{n_{i}+d_{i}}{n_{i}}}{{\sum}n_{i}+1}\rceil>2(\sum n_{i})$.
\end{thm}

The paper is organized as follows. After recalling notation and
definitions we study in detail the singular loci of tangential linear
systems for non tangentially weakly defective varieties. The main
technical result is Theorem~\ref{main theorem} where we prove that,
under suitable hypothesis,
these linear system have not nested  singularities.
This result allow us to apply the standard Noether--Fano inequalities
to show that some tangential projections are not birational, see
Theorem~\ref{Main Theorem NF}.
With this the non identifiability result is at hand following
\cite{Me1}.

\section{Notation}\label{sec:notation} 
We work over the complex field.
Let $X\subset\p^N$ be an irreducible and reduced non-degenerate
variety and  $X^{(h)}$ be the $h$-th symmetric product of $X$.
That is the variety parametrizing unordered sets of $h$ points of $X$. 
Let $U_h^X\subset X^{(h)}$ be the smooth locus, given by sets of $h$
distinct smooth points.
\begin{dfn}
  A point $z\in U^X_h$ represents a
set of $h$ distinct points, say $\{z_1,\ldots, z_h\}$. We say that a point $p\in \p^N$ is in the span of
 $z$, $p\in\langle z\rangle$,  if it is a linear combination of the
 $z_i$.
\end{dfn}
\begin{dfn} The {\it
    abstract $h$-Secant variety} is the irreducible and reduced variety
$$\textit{sec}_{h}(X):=\overline{\{(z,p)\in U_h^X\times\p^N| p\in \Span{z}\}}\subset
X^{(h)}\times\p^N.$$ 

Let $\pi:X^{(h)}\times\p^N\to\p^N$ be the projection onto the second factor. 
The {\it $h$-Secant variety} is
$$\Sec_{h}(X):=\pi(sec_{h}(X))\subset\p^N,$$
and $\pi_h^X:=\pi_{|sec_{h}(X)}:sec_{h}(X)\to\p^N$ is the $h$-secant map of $X$.

The irreducible variety $\textit{sec}_{h}(X)$ has dimension $(hn+h-1)$.
 One says that $X$ is
\textit{$h$-defective}
if $$\dim\Sec_{h}(X)<\min\{\dim\textit{sec}_{h}(X),N\}.$$
For simplicity we will say that $X$ is not defective if it is not
$h$-defective for any $h$.
\end{dfn}

\begin{dfn}
  Let $X\subset\p^N$ be a non-degenerate subvariety. We say that a
  point $p\in\p^N$ has rank $h$ with respect to $X$ if $p\in \langle z\rangle$, for
  some $z\in U_h^X$ and $p\not\in \langle z^\prime\rangle$ for any $z^\prime\in
  U_{h^\prime}^X$, with $h^\prime<h$.
  
We call $g:=:g(X)$ the rank of a general point and we say that
$X\subset\p^N$ is perfect if
$$\frac{N+1}{\dim X+1}\in\N $$
\end{dfn}

\begin{dfn}\label{def:identifiability} A point $p\in\p^N$
  is $h$-identifiable with respect to $X\subset\p^N$ if $p$ is of rank
  $h$ and $(\pi^X_h)^{-1}(p)$ is a single point. The variety $X$ is
  said to be $h$-identifiable if $\pi_h^X$ is a birational map, that
  is the general point of $\Sec_h(X)$ is $h$-identifiable.
  For simplicity we will say that $X\subset\p^N$ is generically identifiable  if the generic
  point of $\p^N$ is $g$-identifiable.
\end{dfn}

\begin{rem}\label{rem:ident_easy}
  Note that $\pi_g^X$ is generically finite if and only if $X$ is
  perfect and not defective.  These are therefore necessary condition
  for identifiability.
\end{rem}

\begin{dfn}
  \label{def:weaklydef} Let $X\subset\p^N$ be a non-degenerate
  variety and $\{x_1,\ldots, x_h\}\subset X$ general points. 
  The variety $X$ is said $h$-weakly defective if the general
  hyperplane singular along $h$ general points is singular along a
  positive dimensional subvariety passing through the points. 
  Let
  $H\in\h(h):=|\I(1)_{x_1^2,\ldots,x_h^2}|$ be a general section, we call
  $\Gamma_h(H)$ its  locus of tangency passing through $x_1,\dots,x_h$. 
\end{dfn}

\begin{dfn}
  For  a linear system $\h$ we set
  $$\Gamma(\h):=\bigcap_{H\in\h} Sing(H) $$
the common singular locus.
\end{dfn}

\begin{rem}
  \label{rem:zeropostive} We want to stress that, by \cite{CC1},  if
  $\Gamma_h(H)$ is zero dimensional in a neighborhood of
  $\{x_1,\ldots,x_h\}$ then $\Gamma_h(H)=\{x_1,\ldots,x_h\}$,
\end{rem}

The notion of tangentially weakly defective
  varieties has been introduced in \cite{CO1}. Here we follow the
  notations of \cite{BBC}.

  For a subset $A=\{x_1,\ldots,x_h\}\subset X$ of general points we set
  $$ M_A := \langle\bigcup_{i}\T_{x_i}X\rangle.$$
By Terracini Lemma  the space $M_A$ is the tangent space to $\Sec_h(X)$
at a general point in $\langle A\rangle$.
\begin{dfn}\label{def: twd}
  The tangential $h$-contact locus $\Gamma_h(A)$ is the closure in $X$ of the union of all the irreducible components which contain at least one point of $A$, of the locus of points of $X$ where $M_A$ is tangent to $X$.
We will write $\gamma_h := \dim \Gamma_h(A)$.
We say that $X$ is $h$-twd ($h$-tangentially weakly defective) if
$\gamma_h > 0$. 
\end{dfn}
\begin{rem}\label{rem:zeropositivetwd}
Note that in general it is difficult to predict the behavior of $\Gamma(\h(h))$ for
non $h$-twd varieties. By definition $\Gamma(\h(h))$ is zero
dimensional in a neighborhood of the assigned singular points but not
much is known about singular components away from these. Our
Proposition~\ref{zero dimensional twd} is a first attempt to study
this problem, under strong hypothesis.  
\end{rem}

For what follows it is useful to introduce also the notion of
tangential projection.
\begin{dfn}
  Let $X\subset\p^N$ be a variety and $A=\{x_1,\ldots,x_h\}\subset X$ a
  set of general points. The $h$-tangential projection (from A) of $X$ is
  $$\tau_h:X\dasharrow\p^M$$
  the linear projection from $M_A$. That is, by Terracini Lemma,   the projection from the
  tangent space of a
  general point $z\in \Span{A}$ of $\Sec_h(X)$ restricted to $X$.
\end{dfn}
\begin{rem}\label{rem:Terracini_tangential}
 By Terracini's Lemma $\tau_h$ is the rational map associated to the
 linear system $\mathcal{H}(h)=|\mathcal{I}_{x_1^2\ldots, x_h^2}(1)|$.
\end{rem}

\section{Properties of contact locus for non twd varieties}

In this section we study properties of the contact loci $\Gamma_{g-1}(H)$ (for  a general $H\in \h(g-1)$)
of projective varieties that are non defective and not $(g-1)$-twd.
In particular in view of applications to Noether--Fano inequalities we
are interested in studying the infinitesimally near singularities of $\h$.




We start recalling 
\cite[Proposition 3.6]{CC1} and its generalization to twd. This
Proposition will be useful to reduce the study of $\Gamma_{g-1}(\h)$
to the special case of $g=2$.
\begin{prop}\label{projection}Let $X\subset\p^N$ be an irreducible and
  reduced non degenerate variety. Assume that $X$ is not $h$-defective
  and $h(\dim X+1)-1< N$. Let $X_s=\tau_s(X)$ be a general tangential
  projection. 
\begin{itemize}
\item[i)] 
$X$ is $h-$weakly
defective if and only if  $X_s$ is $(h-s)-$weakly defective.
\item[ii)] 
$X$ is $h-$twd if and only if  $X_s$ is $(h-s)-$twd.
\end{itemize}
\end{prop}

\begin{proof} Point i) is \cite[Proposition 3.6]{CC1}. 

Point ii) is a simple adaptation of point i) substituting weakly defectivity with twd.
\end{proof}

For future reference we observe the following fact.

\begin{lem}\label{codimension}
Let $Z\subset\mathbb{P}^{n}$ be a  reduced projective
variety of dimension $\dim(Z)=a$. Then
$\codim|\mathcal{I}_{Z}(1)|\geq
a+1$
and equality is fulfilled only by linear spaces.
\end{lem}

\begin{proof} If $Z$ is a linear space there is nothing to
  prove. Assume that $Z$ is not a linear space, then
  $\dim\Span{Z}>\dim Z$.
  We have
  $$\codim|\mathcal{I}_{Z}(1)|=\codim|\mathcal{I}_{\Span{Z}}(1)|=\dim\Span{Z}+1>\dim Z+1.$$
\end{proof}

\begin{dfn}
  Let $X\subset\p^N$ be an irreducible and reduced non-degenerate and
  non $h$-defective variety. Let $\{x_1,\ldots,x_h\}\subset X$ be a
  set of general point and $\h(h)=|\I_{x_1^2,\ldots,x_h^2}(1)|$ the linear
  system of hyperplane sections singular in  $\{x_1,\ldots, x_h\}$.
  Set
  $$\mathcal{W}_h:=\{(H,x)|H\in\h(h), x\in \Gamma_h(H)\}\in\h\times X$$
  and  $\pi^h_{1}:\mathcal{W}_h\rightarrow\mathcal{H}(h)$,  $\pi^h_{2}:\mathcal{W}_h\rightarrow X$
the two canonical projections. 
We  denote with $W_h:=\pi^h_{1}(\mathcal{W}_h)\subset\mathcal{H}(h)$.
\end{dfn}

It is clear that $W_s\subset|\I_{x^2_1,\ldots,x^2_{h}}(1)|$ for any
$h<s$. Then we may identify $W_s$ as a subvariety of $W_h$ for any
$h\leq s$. Our next aim is to prove, in some cases, a more precise result.

\begin{prop}\label{dimension of W} Assume that $X$ is perfect and not
  defective with general rank $g$. Set $\h:=\h(g-2)=|\I_{x_1^2,\ldots,x_{g-2}^2}(1)|$
  and assume
  $$\dim(\Gamma_{g-1}(H))=a,$$
  for $H\in\h(g-1)$.
  Then we have $\codim_{\h(g-2)}(W_{g-1})=a+1$.
\end{prop}

\begin{proof} The variety $X$ is not defective, then $\dim(\h(g-2))=2n+1$.
  By a parameter count we have $\dim\mathcal{W}_{g-1}=2n$. 


By definition for a general $[H]\in W_{g-1}$ we have
$$\dim(\pi_{1}^{-1}(H))=\dim\{x\in X|x\in \Gamma_{g-1}(H)\})=\dim\Gamma_{g-1}(H)$$
therefore we conclude that
$$\dim(W_{g-1})=\dim(\mathcal{W}_{g-1})-\dim(\pi_{1}^{-1}(H))=2n-a$$
yielding $\codim_{\h(g-2)}(W_{g-1})=a+1$.
\end{proof}

The following result is already implicitly used in \cite{CC1} but we
state it as a Proposition for the reader convenience.

 \begin{prop}\label{contact locus linear space}
Let $X\subset\mathbb{P}^{2\dim X+1}$ be an irreducible, reduced non-degenerate
variety. Assume that $X$ is not defective and not $1-$twd.
Then for a general tangent hyperplane $H\in \h(1)$, the tangential
locus $\Gamma_1(H)$ is a
linear space. In particular, under the hypothesis,  also $\Gamma(\h(1))$
is a linear space.
\end{prop}

\begin{proof} If $X$ is not $1$-weakly defective, by Remark
   \ref{rem:zeropostive},  $\Gamma_{1}(H)$ is a point.
   Assume that $X$ is $1$-weakly defective and $\dim \Gamma_1(H)=a$.
 Let $x\in X$ be a general point and $H \in \h(1)$ a general tangent section
 in $x$.
 Let us consider the variety
 $$W_1\subset|\o(1)|=:\h$$
parametrizing singular hyperplane sections.
Proposition \ref{dimension of W} yields $\codim_{\mathcal{H}}(W_1)=a+1$
and so $\codim(\mathbb{T}_{[H]}W_1)=a+1$.
On the other hand, by the infinitesimal Bertini's theorem \cite[Thm
2.2]{CC1},
we have
$$\mathbb{T}_{[H]}W_1\subset\text{\ensuremath{\mathcal{H}}}(-Sing(H))$$
and so $\codim_{\mathcal{H}}(\mathcal{H}(-\Gamma_1(H)))\leq a+1$.

Hence we conclude by Proposition\ref{codimension} that $\Gamma_1(H)$ is a linear space.
\end{proof}

\begin{lem}\label{lemma hyp section weakly}
Let $X\subset\mathbb{P}^{N}$ be an irreducible, reduced non-degenerate
projective variety. Assume that $X$ is $1-$weakly defective with
$\dim(\Gamma_1(H))=a$, for $H\in \h(1)$ a general tangent hyperplane. Then a general hyperplane
section $X'$ of $X$  satisfies
$\dim(\Gamma_1(H'))= a-1$, for $H'$ a general tangent hyperplane to $X'$.
\end{lem}
\begin{proof}
Let $x\in X$ be a general point,  $H\in|\mathcal{I}_{x^{2}}(1)|$
 a general hyperplane section singular at $x$ and  $L\in|\mathcal{I}_{x}(1)|$
 a general hyperplane section passing through $x$.
 The divisor $L$ is smooth in a neighborhood of $x$ and
 $\Bs|\mathcal{I}_{x}(1)|=\{x\}$.  Hence, by Bertini's theorem,
 $$\dim(Sing(H)\cap L))=\dim\Gamma_1(H)-1=a-1$$
To conclude observe that $H_{|L}$ is a general tangent section of $L$ at $x$. 
\end{proof}

Let $(z_1,\dots,z_n)$ be a system of local coordinates at the point
$(x\in X)\cong ((0,\ldots,0)\in\C^n)$.
Every divisor $H\in |\mathcal{I}_{x^2}(1)|$ can be expressed locally as $$H=(Q_{H}(z_{1},\ldots,z_{n})+\sum_{d\geq3}F_{d}(z_{1},\ldots,z_{n})=0)$$ where $Q_{H}(z_1,\dots,z_n)\in \mathbb{C}[z_1,\dots,z_n]_{2}$ is a Quadric and $F_{d}$ are homogeneous polynomials of degree at least
$3$.
The rank of the double point $x\in H$ is by definition the rank of the Quadric $Q_H$.
The singular locus $\mathcal{A}=Sing(Q_H)$ is a linear space $\mathcal{A}\subset \mathbb{C}^n$ of dimension $\dim(\mathcal{A})=\dim(X)-rank(Q_H)$. It is called the asymptotic space of $H$ at the point $x$.  
Let $\nu:X'\rightarrow X$ be the blow up of $X$ at $x$ with
exceptional divisor $E$. Under the identification $E=\mathbb{P}((T_x
X)^{*})=\mathbb{P}^{n-1}$ we have that $\nu_{*}^{-1}(H)\cap
E=\mathbb{P}(Q_{H})$ and $Sing(\nu_{*}^{-1}(H))\cap E\subseteq\mathbb{P}(\mathcal{A})$. 
Note further that  to every point $y\in E$ we can associate uniquely a line $l_y \in T_x X$ corresponding to the tangent direction represented by $y$.

With this notation in mind we are going to improve Proposition \ref{contact locus linear space}.

\begin{prop}\label{Contact locus linear}
Let $X\subset\mathbb{P}^{2 \dim X+1}$ be an irreducible, reduced non-degenerate
projective variety. If $X$ is not defective
$\mathbb{P}(Sing(Q_{H}))=\nu_{*}^{-1}(\Gamma_1(H))\cap E$.
\end{prop}

\begin{proof}
Let $H\in\h(1)$ be a generic
hyperplane section singular at $x$. If $\dim(\Gamma_1(H))=0$, by
\cite[Theorem 1.4] {CC1}, $x$ is an ordinary double point of
$H$. Thus  $Q_{H}$ is a Quadric
of maximal rank.

Assume $\dim(\Gamma_1(H))=a>0$. By Proposition~\ref{contact locus
  linear space} it is enough to prove that $rank (Q_H)=\dim X-a+1$.
Let $\nu:X'\rightarrow X$ be the blow up of $X$
at the general point $x\in X$, with exceptional divisor $E$, and  $H'=\nu^{-1}_{*}(H)$ the strict
transform of $H$. We have
$$\nu_{*}^{-1}(\Gamma_1(H))\cap E\subseteq Sing(H')$$
We already observed that $Sing(H')\cap E\subseteq \mathbb{P}(Sing(Q_{H}))$ hence 
$$\nu_{*}^{-1}(\Gamma_1(H))\cap E\subseteq  \mathbb{P}(Sing(Q_{H})).$$

This leads to
$$rank(Q_{H})\leq \dim(X)-a+1.$$
Let $H_{1},\ldots,H_{a}\in\h(1)$ be general sections. Then Lemma \ref{lemma hyp section weakly} yields that $X^{a}:=H_{1}\cap\ldots\cap H_{a}$ is not
$1-$weakly defective. Hence, by the first part of the proof, we
conclude
$$rank(Q_{H})\geq \dim(X)-a+1$$
and finish the proof.
\end{proof}

We take the opportunity to stress a property of  $\Gamma(\h(g-1))$  for non twd
varieties, recall Remark~\ref{rem:zeropositivetwd}.

\begin{prop}\label{zero dimensional twd}
  Let $X\subset\p^N$ be  a non defective, perfect, irreducible, reduced and non-degenerate
  variety with general rank $g$. Assume that $X$ is not $(g-1)$-twd.
  Then $\Span{\T_{x_1}X,\ldots,\T_{x_{g-1}}X}$ is tangent only along a
  zero dimensional scheme.
\end{prop}
\begin{proof}
  Let $W\subset
  \Span{\T_{x_1}X,\ldots,\T_{x_{g-1}}X}\cap X$ be an irreducible component where
  $\Span{\T_{x_1}X,\ldots,\T_{x_{g-1}}X}$ is tangent to $X$.
  By Proposition~\ref{projection} we have that $X_{g-2}:=\tau_{g-2}(X)$ is not $1$-twd and not defective, where $\tau_{g-2}$ is the linear projection from $\Span{\T_{x_{2}}X,\dots,\T_{x_{g-1}}X}$. 
\begin{claim}
$\tau_{g-2}(W)=\tau_{g-2}(x_1)$
\end{claim}
\begin{proof} Let $y=\tau_{g-2}(x_1)$ and $H\in|\I_{y^2}(1)|$ a general
    tangent hyperplane section.
By Proposition~\ref{projection} $X_{g-2}$ is not $1$-twd and by
Proposition~\ref{contact locus linear space} $\Gamma_1(H)$ is a linear
space, therefore
$$\Gamma_1(H)\cap \T_yX_{g-2}=y.$$
On the other hand, by construction, we have
$$\tau_{g-2}(W)\subset\Gamma_1(H),$$
and this proves the claim.
\end{proof}  

The variety $X$ is not defective and $y=\tau_{g-2}(x_1)$ is a general
point of $X_{g-2}$. Therefore $\tau_{g-2}^{-1}(y)$ is a finite scheme
and we conclude by the Claim that $W$ is $0$-dimensional.
\end{proof}

\begin{remark}
It would be very interesting to understand if the result in
Proposition~\ref{zero dimensional twd} is true for smaller values of
the rank. Unfortunately our proof is
based on Proposition~\ref{contact locus linear space} and cannot be
extended in this direction.
\end{remark}

The following is the main result  of this section.

\begin{thm}\label{main theorem}
Let $X\subseteq\mathbb{P}^{N}$ be a projective irreducible,reduced and non-degenerate
variety of general rank $g$. Let $\{x_{1},...,x_{g-1}\}$ be general
points on $X$ and $\h=\h(g-1)$. Assume that:
\begin{itemize}
\item[-] $X$ is perfect and non defective
\item[-] $X$ is not $(g-1)$-twd
\end{itemize}
Then there is a variety $Y$ and a birational map $\nu:Y\to X$ with the
following property: for any $\epsilon>0$ there is a ${\mathbb Q}$-divisor $D$,
with $D\equiv \nu^{-1}_*\h$ such that for any point $y\in Y$
$$\mult_yD<1+\epsilon.$$
\end{thm}

\begin{proof} The variety $X$ is non defective and not $(g-1)$-twd  Therefore, by
  \cite[Theorem 18]{CM}  the 
  tangential projection $\tau_{g-2}$, from $\{x_2,\ldots,x_{g-1}\}$ is
  birational. Let $X_{g-2}=\tau_{g-2}(X)$ be the image of the
  tangential projection and define $\h':=(\tau_{g-2})_*\h$.
  Then, by Proposition~\ref{projection} 
 $X_{g-2}\in\mathbb{P}^{2\dim X+1}$ is not defective and not $1$-twd.

Let $\sigma:Z\to X_{g-2}$ be the blow up of the point
$\tau_{g-2}(x_1)$, with exceptional divisor $E$ and $\h_Z=\sigma^{-1}_*\h'$.
\begin{claim}\label{cla:emty}  $\Gamma(\h_Z)$ is empty.
\end{claim}
\begin{proof}[Proof of the Claim] By Proposition~\ref{contact locus
  linear space} the tangential locus  $\Gamma_1(H)$ is a linear space. The
variety $X_{g-2}$ is not $1$-twd therefore
$\Gamma(\h')=\tau_{g-2}(x_1)$. This is enough to show that
$\Gamma(\h_Z)\subset E$.

Assume that there is a point  $ z \in\Gamma(\h_Z)\cap E$ and denote
by $l_{z} \subset\p^{2\dim X+1}$ the
corresponding line in the projective space.
By Proposition~\ref{Contact locus linear}  this forces
$l_z\subset\Gamma_1(H)$ and the contradiction
$l_y \in \Gamma(\h')$.
\end{proof}

Let $Y$ be the completion of the Cartesian square
\[
\xymatrix{
  Y \ar@{-->}[rr]^{\eta} \ar[d]_{\nu} && Z \ar[d]^{\sigma}\\
X \ar@{-->}[rr]_{\tau_{g-2}}  && X_{g-2}
}
\]
and $\h_Y=\nu^{-1}_*(\h)=\eta^{-1}_*(\h_Z)$ the strict transform
linear system. By construction and Claim~\ref{cla:emty} the set
$\Gamma(\h_Y)$ is contained in the locus where $\eta$ is not an
isomorphism.  On the
other hand, by monodromy, the same should be true for a different
choice of $(g-2)$ points in $\{x_1,\ldots,x_{g-1}\}$.   Thanks to the
general choice of the points $\{x_1,\ldots,x_{g-1}\}$ this is enough
to show that
$\Gamma(\h_Y)$ is empty.  In particular for any $y\in Y$ there are
divisors $H\in\h_Y$ with $\mult_yH\leq 1$.
To conclude we construct the divisor
$$D=\frac1M\sum_1^M H_i, $$
for $H_i\in\h_Y$ general. The locus $\Gamma(\h_Y)$ is empty therefore we may
assume that for any $y\in Y$ there are at most $\dim\h_Y$ divisors
in $\{H_i\}_{i=1,\ldots,M}$ singular in $y$. This is enough to conclude.    
\end{proof}

\section{Noether--Fano inequalities and generic identifiability} 

In this section we apply the previous results on the singular locus of
linear system $\h(g-1)$ to produce non generic identifiability
statements. 
We start recalling two results in this area.
\begin{thm}[\cite{Me1}]
\label{ident}
Let $X\subseteq\mathbb{P}^{N}$ be a projective, irreducible non-degenerate
variety. Suppose that $X$ is generically identifiable.
Then the $(g(X)-1)$-tangential projection $\tau_{g(X)-1}:X\dasharrow\mathbb{P}^{\dim(X)}$
is birational.
\end{thm}

\begin{thm}[Noether-Fano Inequalities \cite{Co}]
\label{Noether-Fano}
Let $\pi:X\rightarrow X'$ and $\rho:Y\rightarrow Y'$
be two Mori fiber spaces and $\varphi:X\dashrightarrow Y$ a
birational, not biregular,
map

\[
\xymatrix{
  X \ar@{-->}[r]^{\varphi} \ar[d]^{\pi} & Y \ar[d]^{\rho}\\
X' & Y'
}
\]

 Choose a very ample linear system $\mathcal{H}_{Y}$ in $Y$
and let $\mathcal{H}_{X}=\varphi_*^{-1}(\mathcal{H}_{Y})$. Let $a\in\mathbb{Q}$
such that $\mathcal{H}_{X}\equiv-aK_{X}+\pi^{*}(A)$ for some divisor
$A\in Pic(S)$.

Then either $(X,\frac{1}{a}\mathcal{H}_{X})$ has not canonical singularities
or $K_{X}+\frac{1}{a}\mathcal{H}_{X}$ is not NEF.
\end{thm}

We are ready to connect the contact loci properties and the
Noether--Fano inequalities to produce a tool for non identifiability statements.

\begin{thm}\label{Main Theorem NF}
Let $X^{n}\subset\mathbb{P}^{N}$ be a projective smooth non-degenerate variety and $\tau_{g-1}:X\dashrightarrow\mathbb{P}^{\dim X}$
be a general tangential projection, associated to the linear system $\h:=\h(g-1)$. Assume that
\begin{itemize}
\item[-]
$\pi:X\rightarrow S$ is a structure of a Mori fiber space such
that $$\mathcal{H}\equiv_{\pi}-aK_{X}+\pi^{*}(A)$$ with $a> 1$ a rational
number and $A\in Pic(S)$ 
\item[-]
The $\mathbb{Q}$-divisor $K_{X}+\frac{1}{a}\mathcal{H}$
is NEF
\item[-]
$X$ is not $(g-1)-$twd
\end{itemize}

Then $\tau_{g-1}$ is not birational, in particular $X$ is not
generically identifiable.
\end{thm}

\begin{proof} If $\pi^X_g:sec_g(X)\to\p^N$ is of fiber type  then $\tau_{g-1}$ is of fiber type,
  see for instance \cite[Lemma16 (i)]{CM}, and we conclude, by
  Theorem~\ref{ident}, that $X$ is not identifiable.

  Then we may assume that $X$
  is perfect and not defective. In particular $\tau_{g-1}$ is a not biregular map
  onto $\p^{\dim X}$.

By Theorem~\ref{main theorem} there
  is a  variety $Y$ and a birational map $\nu:Y\to X$ with the
following property: for any $\epsilon>0$ there is a ${\mathbb Q}$-divisor $D$,
with $D\equiv \nu^{-1}_*\h(g-1)$ such that for any point $y\in Y$
$$\mult_yD<1+\epsilon.$$
In particular $(Y,\frac1a\nu^{-1}_*(\h(g-1)))$ and henceforth
$(X,\frac1a\h(g-1))$ have canonical singularities. Then, by
Theorem~\ref{Noether-Fano}  applied to the diagram
\[
\xymatrix{
  X \ar@{-->}[rr]^{\tau_{g-1}} \ar[d]^{\pi} && \p^{\dim X} \ar[d]\\
S && Spec{\mathbb C}
},
\]
the map  $\tau_{g-1}$ cannot be birational and therefore $X$
is not generically identifiable by Theorem~\ref{ident}.
\end{proof}

We are ready to prove the non identifiability statement announced
in the introduction.

\begin{dfn}
  Let $\textbf{n}=(n_1,\dots,n_r)$ and $\textbf{d}=(d_1,\dots,d_r)$ be two $r-$uples of positive integers.
The Segre-Veronese variety $SV^{\textbf{\textit{n}}}_{\textbf{\textit{d}}}$ is
the embedding of $$\mathbb{P}^{n_{1}}\times\dots\times\mathbb{P}^{n_{r}}\subset\mathbb{P}^{{\prod}\binom{n_{i}+d_{i}}{n_{i}}-1}$$ via
the complete linear system $|\pi_{1}^{*}\mathcal{O}_{\mathbb{P}^{n_{1}}}(d_{1})\otimes...\otimes\pi_{r}^{*}\mathcal{O}_{\mathbb{P}^{n_{r}}}(d_{r})|$ where $$\pi_{i}:\p^{n_{1}}\times\dots\times\p^{n_{r}}\rightarrow \p^{n_{i}}$$ are the canonical projections.
\end{dfn}

\begin{thm}Fix two multiindexes  $\textbf{\textit{n}}=(n_1,\dots,n_r)$
  and $\textbf{\textit{d}}=(d_1,\dots,d_r)$. Let
  $X=SV^{\textbf{\textit{n}}}_{\textbf{\textit{d}}}$ the corresponding
  Segre-Veronese variety. Assume that $d_i> n_i+1$, for $i=1,\ldots, r$,
  and  $$\lceil\frac{{\prod}\binom{n_{i}+d_{i}}{n_{i}}}{{\sum}n_{i}+1}\rceil>2(\sum n_{i}).$$
Then $X$ is not generically identifiable.
\end{thm}

\begin{proof} If $X$ is defective or non perfect the statement is
  clear, recall Remark~\ref{rem:ident_easy}. Assume that
  $X$ is not defective and perfect. Then
  $\tau_{g-1}:X\dasharrow\p^{\dim X}$ is generically finite. The numerical
  assumption
  reads
  $$g(X)=\lceil\frac{{\prod}\binom{n_{i}+d_{i}}{n_{i}}}{{\sum}n_{i}+1}\rceil>2(\sum n_{i})=2\dim(X) $$
and,  by \cite[Corollary 22]{CM}, the variety  $X$
is not $(g-1)-$twd.

After reordering the indexes we may assume that
\begin{equation}
  \label{eq:NEF}
  \frac{n_1+1}{d_1}\geq \frac{n_i+1}{d_i}, \mbox{ for any $i$}.
\end{equation}

Let $p:X\to Y$ be the canonical projection onto the
Segre-Veronese $Y=SV^{(n_2,\ldots,n_r)}_{(d_2,\ldots,d_r)}$ and
$a=\frac{d_1}{n_1+1}>1$.
Then $p$ is a Mori fiber Space and
$$K_{X}+\frac{1}{a}\mathcal{H}(g-1){\equiv_p} 0.$$
Further note that the cone of effective divisor of $X$ is spanned by
the lines in the factors $\p^{n_i}$ and, by Equation~(\ref{eq:NEF}), we have
$$ K_{X}+\frac{1}{a}\mathcal{H}(g-1)\cdot
l_i=-(n_i+1)+\frac{n_1+1}{d_1}d_i\geq 0,$$
This shows that $ K_{X}+\frac{1}{a}\mathcal{H}(g-1)$ is NEF and, by
Theorem~\ref{Main Theorem NF},
 we prove  that $X$ is not generically identifiable.
\end{proof}

\begin{remark}
In recent years the Secant varieties of Segre-Veronese varieties have
been studied intensively,  see for instance
\cite{AB},\cite{AMR},\cite{BBC1},\cite{BCC}. 
However, to the best of our knowledge, this is the first  result
regarding non generic identifiability for infinite classes of Segre-Veronese varieties with $r\geq2$.
\end{remark}

\end{document}